\RequirePackage{fix-cm}
\documentclass[smallextended]{svjour3}       
\smartqed  
\usepackage{graphicx}
\usepackage{amsfonts}
\usepackage{bbm}
\usepackage{amsfonts,epsfig}
\begin{document}

\title{New error bounds for linear complementarity problems of Nekrasov matrices and $B$-Nekrasov matrices
\thanks{This work is partly supported by National Natural Science
Foundations of China (11361074), Natural Science Foundations of
Zhejiang Province of China ( LY14A010007) and Natural Science
Foundations of Fujian Province (2016J01028)} }


\author{Chaoqian Li \and Pingfan Dai \and
        Yaotang Li 
}


\institute{Chaoqian Li  \and Yaotang Li \at              School of
Mathematics and Statistics, Yunnan
University, Kunming, Yunnan, 650091, PR China \\
              \email{lichaoqian@ynu.edu.cn,~liyaotang@ynu.edu.cn}           
         \and
        Pingfan Dai          \at School of Mathematics and Statistics, Xian Jiaotong University, Xian, Shaanxi 710049,
P. R. China \at  Department of Information Engineering, Sanming
University, Sanming, Fujian 365004, P. R. China
             \email{ daipf2004@163.com}
}

\date{Received: date / Accepted: date}

\maketitle

\begin{abstract}
New error bounds for the linear complementarity problems are given
respectively  when the involved matrices are Nekrasov matrices and
$B$-Nekrasov matrices. Numerical examples are given to show that the
new bounds are better respectively than those provided by
Garc\'ia-Esnaola and Pe\~na in [15,16] in some cases.
\end{abstract}

\section{Introduction}
\label{intro} Linear complementarity problem  LCP$(M, q)$ is to find
a vector $x\in R^{n}$ such that
\begin{equation}\label{neq1.1}
x\geq 0, Mx+q\geq 0, (Mx+q)^Tx=0 \end{equation} or to show that no
such vector $x$ exists, where $M =[m_{ij}]\in R^{n\times n}$ and
$q\in R^n$. The  LCP$(M, q)$  has various applications in the Nash
equilibrium point of a bimatrix game, the contact problem and the
free boundary problem  for journal bearing, for details,  see
\cite{Ber,Co,Mu}.

 The LCP$(M, q)$ has a unique
solution for any $q\in R^n$ if and only if $M$ is a $P$-matrix
\cite{Co}. We here say a matrix $M \in R^{n, n}$ is a $P$-matrix if
all its principal minors are positive. In \cite{Chen0}, Chen and
Xiang gave the following error bound of the LCP$(M, q)$ when $M$ is
a $P$-matrix:
\[ ||x-x^*||_{\infty} \leq \max\limits_{d\in [0,1]^n}||(I -D+DM)^{-1}||_{\infty} ||r(x)||_\infty,\]
where $x^*$ is the solution of the LCP$(M, q)$, $r(x)=\min\{
x,Mx+q\}$, $D=diag(d_i)$ with $0\leq d_i \leq 1$,
$d=[d_1,\cdots,d_n]^T \in [0,1]^n$ denotes $0\leq d_i\leq 1$ for
each $i\in N$, and the min operator $r(x)$ denotes the componentwise
minimum of two vectors. Furthermore, if $M$ is a certain structure
matrix, such as an $H$-matrix with positive diagonals
\cite{Chen0,Chen,Ga0,Ga1,Ga3}, a $B$-matrix \cite{Pena1,Ga}, a
$DB$-matrix \cite{Dai}, an $SB$-matrix \cite{Dai1,Dai2}, a
$B^S$-matrix \cite{Ga2}, an $MB$-matrix \cite{Che}, and a
$B$-Nekrasov matrix \cite{Ga4}, then some corresponding results on
the bound of $\max\limits_{d\in [0,1]^n}||(I -D+DM)^{-1}||_{\infty}$
can be derived; for details, see
\cite{Che,Chen0,Chen,Dai,Dai1,Dai2,Dai3,Ga0,Ga1,Ga2,Ga3}.

In this paper, we  focus on the bound of $\max\limits_{d\in
[0,1]^n}||(I -D+DM)^{-1}||_{\infty}$, and  give its new bounds when
$M$ is a Nekrasov matrix with positive diagonals and  a $B$-Nekrasov
matrix, respectively. Numerical examples are given to show the new
bounds are respectively better than those in \cite{Ga3} and
\cite{Ga4} in some cases.
\section{Error bounds for linear complementarity problems of Nekrasov matrices}
\label{sec:1} Garc\'{i}a-Esnaola and  Pe\~na in \cite{Ga3} provided
the following bound for  $\max\limits_{d\in [0,1]^n}||(I
-D+DM)^{-1}||_{\infty}$, when $M$ is a Nekrasov matrix with positive
diagonals. Here, a matrix $A=[a_{ij}]\in C^{n, n}$ is called a
Nekrasov matrix \cite{Ko,Li0} if for each $i\in N=\{1,2,\ldots,n\}$,
\[|a_{ii}|>h_i(A),\]
where $h_1(A)=\sum\limits_{j\neq 1} |a_{1j}|$ and
$h_i(A)=\sum\limits_{j=1}^{i-1}\frac{|a_{ij}|}{|a_{jj}|}h_j(A)+\sum\limits_{j=i+1}^{n}|a_{ij}|$,
$i=2,3,\ldots,n$.

\begin{theorem}\emph{\cite[Theorem 3]{Ga3}}  \label{G-Nek}  Let $M=[m_{ij}]\in R^{n, n}$
be a Nekrasov matrix with $m_{ii}>0$ for $i\in N$ such that for each
$i=1,2,\ldots, n-1,$ $m_{ij}\neq 0$ for some $j>i$. Let $W =
diag(w_1,\cdots,w_n )$ with $w_i=\frac{h_i(M)}{m_{ii}}$ for
$i=1,2\ldots,n-1$ and $w_n=\frac{h_n(M)}{m_{nn}}+\varepsilon$,
$\varepsilon \in \left(0, 1- \frac{h_n(M)}{m_{nn}}\right)$. Then
\begin{equation} \label{Bound_N1}\max\limits_{d\in [0,1]^n}||(I-D+DM)^{-1}||_{\infty} \leq
\max\left\{\frac{\max\limits_{i\in N}w_i}{\min\limits_{i\in N}s_i},
\frac{\max\limits_{i\in N}w_i}{\min\limits_{i\in N}w_i}
\right\},\end{equation} where for each $i=1,2,\ldots,n-1$,
$s_i=\sum\limits_{j=i+1}^n|m_{ij}| (1-w_j)$ and $s_n=\varepsilon
m_{nn}$.
\end{theorem}

It is not difficult to see that when $M=[m_{ij}]\in R^{n, n}$ is a
Nekrasov matrix with $m_{ij}= 0$ for any $j>i$ and for some $i\in
\{1,2,\ldots,n-1\}$, Theorem \ref{G-Nek} cannot be used to estimate
$\max\limits_{d\in [0,1]^n}||(I-D+DM)^{-1}||_{\infty}$, and that
when $ \varepsilon \rightarrow 0$,
\[ s_n=\varepsilon
m_{nn} \rightarrow 0 ~and ~\min\limits_{i\in N}s_i  \rightarrow 0,\]
which gives  the bound
\[\max\left\{\frac{\max\limits_{i\in N}w_i}{\min\limits_{i\in N}s_i},
\frac{\max\limits_{i\in N}w_i}{\min\limits_{i\in N}w_i}
\right\}\rightarrow +\infty. \] These facts show that in some cases
the bound in Theorem \ref{G-Nek} is not always effective to estimate
$\max\limits_{d\in [0,1]^n}||(I-D+DM)^{-1}||_{\infty}$ when  $M$ is
a Nekrasov matrix with positive diagonals. To conquer these two
drawbacks, we next give a new bound which only depends on the
entries of $M$. Before that, some results on Nekrasov matrices which
will be used later are given as follows.

\begin{lemma} \label{lem1} Let $M=[m_{ij}]\in C^{n, n}$
be a Nekrasov matrix with $m_{ii}>0$ for $i\in N$ and let
$\tilde{M}=I-D+DM=[\tilde{m}_{ij}]$ where  $D=diag(d_i)$ with $0\leq
d_i \leq 1$. Then $\tilde{M}$ is a Nekrasov matrix. Furthermore, for
each $i\in N$, \begin{equation} \label{neq2.1}
\frac{h_i(\tilde{M})}{\tilde{m}_{ii}}\leq
\frac{h_i(M)}{m_{ii}}.\end{equation}
\end{lemma}

\begin{proof} We prove that (\ref{neq2.1}) holds by mathematical induction, and then (\ref{neq2.1})
immediately implies that  $\tilde{M}$  is a Nekrasov matrix. Note
that
\[\tilde{m}_{ij}=\left\{ \begin{array}{cc}
  1-d_i+d_im_{ij},  &i=j,\\
 d_im_{ij},  &i\neq j.
\end{array} \right.\] Hence, for each $i\in N$,
\[ \frac{d_i}{\tilde{m}_{ii}}= \frac{d_i}{1-d_i+d_im_{ii}} \leq  \frac{1}{m_{ii}}, ~for ~0\leq d_i \leq 1, ~i\in
N.\] Then we have that for  $i=1$,
\begin{eqnarray*}
\frac{h_{1}(\tilde{M})}{\tilde{m}_{11}}&=&
\frac{d_1\sum\limits_{j\neq 1}|m_{ij}|}{1-d_1+m_{11}d_1}\\&\leq &
\frac{\sum\limits_{j\neq 1}|m_{ij}|}{m_{11}}\\
&=&\frac{h_1(M)}{m_{11}}.\nonumber\end{eqnarray*} Now suppose that
(\ref{neq2.1}) holds  for $i=2,3,\ldots,k$ and $k<n$. Since
\begin{eqnarray*} \frac{h_{k+1}(\tilde{M})}{\tilde{m}_{k+1,k+1}}&=&\frac{
\sum\limits_{j=1}^{k}|\tilde{m}_{k+1,j}|\frac{h_{j}(\tilde{M})}{\tilde{m}_{jj}}+\sum\limits_{j=k+2}^{n}
|\tilde{m}_{k+1,j}|}{\tilde{m}_{k+1,k+1}}
\\&\leq& \frac{\sum\limits_{j=1}^{k}|\tilde{m}_{k+1,j}|\frac{h_{j}(M)}{m_{jj}}+\sum\limits_{j=k+2}^{n}
|\tilde{m}_{k+1,j}|}{\tilde{m}_{k+1,k+1}}
 \\&=& \frac{d_{k+1}\left(\sum\limits_{j=1}^{k}|m_{k+1,j}|\frac{h_{j}(M)}{m_{jj}}+\sum\limits_{j=k+2}^{n}
|m_{k+1j}|\right)}{1-d_{k+1}+m_{k+1,k+1}d_{k+1}}\\&\leq&\frac{
\sum\limits_{j=1}^{k}|m_{k+1,j}|\frac{h_{j}(M)}{m_{jj}}+\sum\limits_{j=k+2}^{n}
|m_{k+1j}|  }{m_{k+1,k+1}}\\
\\ &=& \frac{h_{k+1}(M)}{m_{k+1,k+1}},\nonumber\end{eqnarray*} by mathematical induction we have that for each $i\in N$,
(\ref{neq2.1}) holds. Furthermore, the fact that $M$ is a Nekrasov
matrix yields
\[\frac{h_i(M)}{m_{ii}}<
1 ~for ~each~i\in N.\] By (\ref{neq2.1}) we can conclude
that\[\frac{h_i(\tilde{M})}{\tilde{m}_{ii}}< 1 ~for ~each~i\in N,\]
equivalently, $ |\tilde{m}_{ii}|>h_{i}(\tilde{M})$ for each $i\in
N$, consequently, $\tilde{M}$ is a Nekrasov matrix.
\end{proof}

\begin{lemma} \cite[Lemma 3]{Li} \label{lem2}
Let $\gamma > 0$ and $ \eta \geq 0 $. Then for any $x\in [0,1]$,
\begin{equation} \label{eq4} \frac{1}{1-x+\gamma x} \leq  \frac{1}{\min\{\gamma,1\}}\end{equation} and
\begin{equation} \label{eq5} \frac{\eta x}{1-x+\gamma x} \leq  \frac{\eta }{\gamma}.\end{equation}
\end{lemma}

Lemma \ref{lem2} will be used in the proofs of the following lemma
and of Theorem \ref{main1}.

\begin{lemma} \label{lem3} Let $M=[m_{ij}]\in C^{n, n}$
be a Nekrasov matrix with $m_{ii}>0$ for $i\in N$ and let
$\tilde{M}=I-D+DM=[\tilde{m}_{ij}]$ where  $D=diag(d_i)$ with $0\leq
d_i \leq 1$. Then
\begin{equation} \label{nneq2.2} z_i(\tilde{M})\leq
\eta_i(M)\end{equation} and
\begin{equation} \label{neq2.2}
\frac{z_i(\tilde{M})}{\tilde{m}_{ii}}\leq
\frac{\eta_i(M)}{\min\{m_{ii},1\}},\end{equation} where
$z_1(\tilde{M})=\eta_1(M)=1$,
$z_i(\tilde{M})=\sum\limits_{j=1}^{i-1}
\frac{|\tilde{m}_{ij}|}{|\tilde{m}_{jj}|}z_j(\tilde{M})+1,$ and \[
\eta_i(M)=\sum\limits_{j=1}^{i-1}
\frac{|m_{ij}|}{\min\{|m_{jj}|,1\}}\eta_j(M)+1, ~i=2,3\ldots,n.\]
\end{lemma}

\begin{proof} We only prove (\ref{nneq2.2}), and (\ref{neq2.2}) follows from the fact that
\[\frac{1}{\tilde{m}_{ii}}= \frac{1}{1-d_i+d_im_{ii}} \leq \frac{1}{\min\{m_{ii},1\}}~for ~i\in N. \]
Note that
\[z_1(\tilde{M})\leq \eta_1(M).\] We
now suppose that (\ref{nneq2.2}) holds  for $i=2,3,\ldots,k$ and
$k<n$. Since
\begin{eqnarray*}
z_{k+1}(\tilde{M})&=&
\sum\limits_{j=1}^{k}|\tilde{m}_{k+1,j}|\frac{z_j(\tilde{M})}{|\tilde{m}_{jj}|}+1\\
&\leq&\sum\limits_{j=1}^{k}|\tilde{m}_{k+1,j}|\frac{\eta_j(M)}{\min\{m_{jj},1\}}+1\\
&=&d_{k+1}\sum\limits_{j=1}^{k}|m_{k+1,j}|\frac{\eta_j(M)}{\min\{m_{jj},1\}}+1\\
&\leq&\sum\limits_{j=1}^{k}|m_{k+1,j}|\frac{\eta_j(M)}{\min\{m_{jj},1\}}+1\\
&=&\eta_{k+1}(M),\nonumber\end{eqnarray*} by mathematical induction
we have that for each $i\in N$, (\ref{nneq2.2}) holds.
\end{proof}

\begin{lemma} \emph{\cite[Theorem 2]{Ko}}\label{lem4}
Let $A=[a_{ij}]\in C^{n, n}$ be a Nekrasov matrix. Then
\begin{equation} \label{neq2.3} ||A^{-1}||_\infty \leq \max\limits_{i\in
N}\frac{z_i(A)}{|a_{ii}|-h_i(A)},\end{equation}  where $z_1(A)=1$
and $z_i(A)=\sum\limits_{j=1}^{i-1}
\frac{|a_{ij}|}{|a_{jj}|}z_j(A)+1, i=2,3\ldots,n.$
\end{lemma}

By Lemmas \ref{lem1}, \ref{lem2}, \ref{lem3} and \ref{lem4}, we can
obtain  the following bound for $\max\limits_{d\in
[0,1]^n}||(I-D+DM)^{-1}||_{\infty}$ when $M$ is a Nekrasov matrix.

\begin{theorem} \label{main1} Let $M=[m_{ij}]\in R^{n, n}$ be a Nekrasov matrix with $m_{ii}>0$ for $i\in N$ and let $\tilde{M}=I-D+DM$ where  $D=diag(d_i)$ with $0\leq d_i \leq 1$. Then
\begin{equation} \label{bound_N2} \max\limits_{d\in [0,1]^n}||\tilde{M}^{-1}||_{\infty} \leq
\max\limits_{i\in N}
\frac{\eta_i(M)}{\min\{m_{ii}-h_i(M),1\}},\end{equation} where
$\eta_i(M)$ is defined in Lemma \ref{lem3}.
\end{theorem}

\begin{proof} Let $\tilde{M}=I-D+DM=[\tilde{m}_{ij}]$. By Lemma \ref{lem1} and Lemma \ref{lem4}, we
have that $ \tilde{M}$ is a Nekrasov matrix, and
\begin{equation} \label{Bod1}  ||\tilde{M}^{-1}||_\infty \leq  \max\limits_{i\in
N}\frac{z_i(\tilde{M})}{\tilde{m}_{ii}-h_i(\tilde{M})}.\end{equation}
Note that
\begin{eqnarray*}
\frac{z_1(\tilde{M})}{\tilde{m}_{11}-h_1(\tilde{M})}&=& \frac{1
}{\tilde{m}_{11}-\sum\limits_{j=2}^n |\tilde{m}_{1j}|}\\&=& \frac{1
}{1-d_1+m_{11}d_1-\sum\limits_{j=2}^n |m_{1j}|d_1}\\&\leq &
\frac{1}{\min\{m_{11}- \sum\limits_{j=2}^n
|m_{1j}|,1\}}\\&=&\frac{\eta_1(M)}{\min\{m_{11}- h_1(M),1\}}
\end{eqnarray*} and for $i=2,3,\ldots,n$, we have by Lemma \ref{lem3} and (\ref{neq2.1}) that
\begin{eqnarray*}
\frac{z_i(\tilde{M})}{\tilde{m}_{ii}-h_i(\tilde{M})}&=&
\frac{\sum\limits_{j=1}^{i-1} |\tilde{m}_{ij}|
\frac{z_j(\tilde{M})}{\tilde{m}_{jj}}+1
}{\tilde{m}_{ii}-\left(\sum\limits_{j=1}^{i-1}
|\tilde{m}_{ij}|\frac{h_j(\tilde{M})}{\tilde{m}_{jj}}+
\sum\limits_{j=i+1}^{n} |\tilde{m}_{ij}|  \right)}\\&\leq&
\frac{\sum\limits_{j=1}^{i-1} |\tilde{m}_{ij}|
\frac{\eta_j(M)}{\min\{m_{jj},1\}}+1
}{\tilde{m}_{ii}-\left(\sum\limits_{j=1}^{i-1}
|\tilde{m}_{ij}|\frac{h_j(M)}{m_{jj}}+ \sum\limits_{j=i+1}^{n}
|\tilde{m}_{ij}|  \right)}
\\&= &\frac{d_i\sum\limits_{j=1}^{i-1} |m_{ij}|
\frac{\eta_j(M)}{\min\{m_{jj},1\}}+1
}{1-d_i+m_{ii}d_i-d_i\left(\sum\limits_{j=1}^{i-1}
|m_{ij}|\frac{h_j(M)}{m_{jj}}+ \sum\limits_{j=i+1}^{n} |m_{ij}|
\right)}\\&\leq& \frac{\sum\limits_{j=1}^{i-1} |m_{ij}|
\frac{\eta_j(M)}{\min\{m_{jj},1\}}+1
}{1-d_i+m_{ii}d_i-d_i\left(\sum\limits_{j=1}^{i-1}
|m_{ij}|\frac{h_j(M)}{m_{jj}}+ \sum\limits_{j=i+1}^{n} |m_{ij}|
\right)}\\&\leq& \frac{\eta_i(M) }{\min\{
 m_{ii}-h_i(M),1\} }.
\end{eqnarray*}
Therefore,  by (\ref{Bod1}) we have \[ ||\tilde{M}^{-1}||_\infty
\leq \max\limits_{i\in
N}\frac{z_i(\tilde{M})}{\tilde{m}_{ii}-h_i(\tilde{M})}\leq
\max\limits_{i\in N} \frac{\eta_i(M)}{\min\{m_{ii}-h_i(M),1\}}.
\] The conclusion follows.
\end{proof}

Remark here that when $m_{ii}=1$ for all $i\in N$ in Theorem
\ref{main1}, then
\[\min\{m_{ii}-h_i(M),1\}= 1-h_i(M),\] which yields the following
result.

\begin{corollary} \label{cor1} Let $M=[m_{ij}]\in R^{n, n}$ be a Nekrasov matrix with $m_{ii}=1$ for $i\in N$
and let $\tilde{M}=I-D+DM$ where  $D=diag(d_i)$ with $0\leq d_i \leq
1$. Then
\[ \max\limits_{d\in [0,1]^n}||\tilde{M}^{-1}||_{\infty} \leq
\max\limits_{i\in N} \frac{\eta_i(M)}{1-h_i(M)}.\]
\end{corollary}

\begin{example} \label{ex1} Consider the following matrix \[M =
\left[ \begin{array}{cccc}
  5  &-\frac{1}{5}   &-\frac{2}{5} &-\frac{1}{2}\\
 -\frac{1}{10}  &2  &- \frac{1}{2} & -\frac{1}{10}  \\
 -\frac{1}{2}  &-\frac{1}{10}   &1.5 &-\frac{1}{10}\\
 -\frac{2}{5} & -\frac{2}{5} &-\frac{4}{5} & 1.2
\end{array} \right].\]
By computations,
\[ h_1(M)=1.1000 < |m_{11}|,  h_2(M)=0.6220< |m_{22}|,\]
\[ h_3(M)=0.2411< |m_{33}|, ~and~ h_4(M)=0.3410< |m_{44}|.\]
Hence, $M$ is a Nekrasov matrix. The diagonal matrix $ W$ in Theorem
\ref{G-Nek} is given by \[ W=diag\left( 0.2200, 0.3110, 0.1607,
0.2842+\varepsilon \right)\] with $ \varepsilon \in (0,0.7158)$.
Hence, by Theorem \ref{G-Nek} we can get the bound (\ref{Bound_N1})
involved with $\varepsilon\in (0,0.7158) $ for $ \max\limits_{d\in
[0,1]^n}||(I-D+DM)^{-1}||_{\infty}$, which is drawn in Figure
\ref{Fig1}. Furthermore, by Theorem \ref{main1}, we can get that the
bound (\ref{bound_N2}) for $ \max\limits_{d\in
[0,1]^n}||(I-D+DM)^{-1}||_{\infty}$ is $3.6414$. It is easy to see
from Figure \ref{Fig1} that the bound in Theorem \ref{main1} is
smaller than that in Theorem \ref{G-Nek} (Theorem 3 in \cite{Ga3})
in some cases.

\begin{figure}[hbtp]\label{Fig1} \centering
\includegraphics[width=4in]{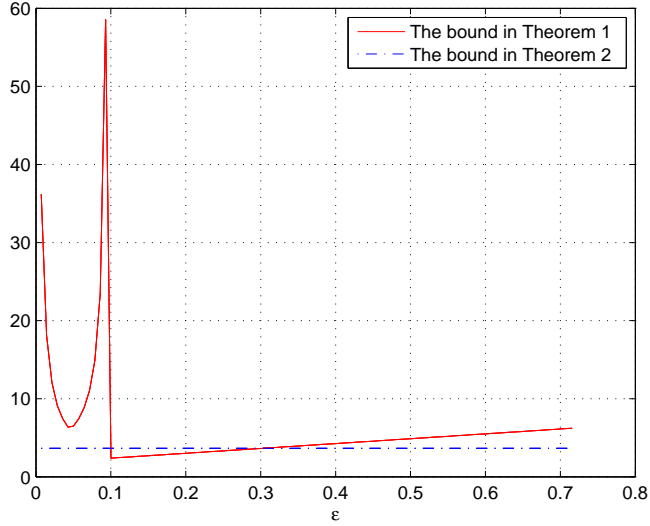}
\caption{The bounds in Theorems \ref{G-Nek} and \ref{main1}}
\end{figure}
\end{example}

\begin{example} \label{ex2} Consider the following Nekrasov  matrix \[M =
\left[ \begin{array}{cccc}
 1 &-\frac{2}{5}   &-\frac{2}{5} &0 \\
 -\frac{1}{2}  &1    &-\frac{1}{4} &-\frac{1}{4}   \\
 -\frac{2}{5}  &-\frac{2}{5}    &1 &0 \\
 -\frac{1}{5} & -\frac{2}{5}  &-\frac{2}{5}  & 1
\end{array} \right].\]
Since $ m_{34}=0$, we cannot use the bound (\ref{Bound_N1})  in
Theorem \ref{G-Nek}. However, by Theorem \ref{main1}, we have
\[\max\limits_{d\in [0,1]^n}||(I-D+DM)^{-1}||_{\infty}\leq 15.\]
\end{example}

\section{Error bounds for linear complementarity problems of $B$-Nekrasov matrices }
\label{sec:2} The class of $B$-Nekrasov matrices is introduced by
Garc\'ia-Esnaola and Pe\~na \cite{Ga4} as a subclass of
$P$-matrices. We say that $M$ is a $B$-Nekrasov matrix if it can be
written as
\begin{equation}\label{dec} M =B^++C,\end{equation} where
\[ B^+ =[b_{ij}]= \left[ \begin{array}{ccc}
 m_{11}-r_1^+    &\cdots     &m_{1n}-r_1^+ \\
 \vdots          &          &\vdots  \\
 m_{n1}-r_n^+     &\cdots    &m_{nn}-r_n^+
\end{array} \right], ~and~C=\left[ \begin{array}{ccc}
 r_1^+    &\cdots     &r_1^+ \\
 \vdots          &          &\vdots  \\
r_n^+     &\cdots    &r_n^+
\end{array} \right] \]
with $r_i^+=\max\{ 0,m_{ij}|j\neq i\}$ and $B^+ $ is a Nekrasov
matrix whose diagonal entries are all positive. Obviously, $B^+$ is
a $Z$-matrix and $C$ is a nonnegative matrix of rank $1$
\cite{Ber,Ga4}. Also in \cite{Ga4}, Garc\'ia-Esnaola and Pe\~na
provided the following error bound for LCP($M,q$) when $M$ is a
$B$-Nekrasov matrix.

\begin{theorem} \emph{\cite[Theorem 2]{Ga4}}
\label{B-N} Let $M=[m_{ij}]\in R^{n, n}$ be a $B$-Nekrasov matrix
such that for each $i=1,2,\ldots,n-1$  there exists $k>i$ with
$m_{ik}<\max \{0, m_{ij}|j\neq i\}=r_i^+,$ let $B^+$ be the matrix
of (\ref{dec}) and let $W = diag(w_1,\cdots,w_n )$ with
$w_i=\frac{h_i(B^+)}{m_{ii}-r_i^+}$ for $i=1,2\ldots,n-1$ and
$w_n=\frac{h_n(B^+)}{m_{nn}-r_n^+}+\varepsilon$, $\varepsilon \in
\left(0, 1- \frac{h_n(B^+)}{m_{nn}-r_n^+}\right)$, such that $
\bar{B}=B^+ W=[\bar{b}_{ij}]$ is a strictly diagonally dominant
$Z$-matrix. Let $ \beta_i= \bar{b}_{ii}-\sum\limits_{j\neq i}
|\bar{b}_{ij}|$ and $ \delta_i=\frac{\beta_i}{w_i}$ for $i\in N$,
and $\delta=\min\limits_{i\in N} \delta_i $. Then \begin{equation}
\label{bound_B_N1} \max\limits_{d\in
[0,1]^n}||(I-D+DM)^{-1}||_{\infty} \leq \frac{(n-1)\max\limits_{i\in
N} w_i}{\min\{\delta, 1\}\min\{w_i\}}.
\end{equation} \end{theorem}

Remark here that the bound (\ref{bound_B_N1}) in Theorem \ref{B-N}
has some drawbacks because it is involved with a parameter
$\varepsilon $ in the interval $\left(0, 1-
\frac{h_n(B^+)}{m_{nn}-r_n^+}\right)$ and it is not easy to decide
the optimum value of $\varepsilon $ in general. Based on the results
obtained in Section 2, we next give a new bound, which only depends
on the entries of $M $, for $\max\limits_{d\in
[0,1]^n}||(I-D+DM)^{-1}||_{\infty}$ when $ M$ is a $B$-Nekrasov
matrix.

\begin{theorem}\label{main2}  Let $M=[m_{ij}]\in R^{n\times n}$ be a $B$-Nekrasov matrix, and let $B^+=[b_{ij}]$ be the matrix
of (\ref{dec}).  Then \begin{equation} \label{bound_B_N2}
\max\limits_{d\in [0,1]^n}||(I-D+DM)^{-1}||_{\infty} \leq
\max\limits_{i\in N}
\frac{(n-1)\eta_i(B^+)}{\min\{b_{ii}-h_i(B^+),1\}}, \end{equation}
where $\eta_1(B^+)=1$, and \[ \eta_i(B^+)=\sum\limits_{j=1}^{i-1}
\frac{|b_{ij}|}{\min\{b_{jj},1\}}\eta_j(B^+)+1, ~i=2,3\ldots,n.\]
\end{theorem}

\begin{proof} Since $M$ is a $B$-Nekrasov matrix, $M = B^+ + C $ as in (\ref{dec}), with $B^+$ being a Nekrasov $Z$-matrix with positive
diagonal entries. Given a diagonal matrix $D = diag(d_i ),$ with $0
\leq d_i \leq 1$, we have $\tilde{M}= I -D + DM= (I-D + DB^+ ) + DC
= \tilde{B}^+ + \tilde{C},$ where $\tilde{B}^+ = I -D + DB^+ $ and
$\tilde{C} = DC$. By Theorem 2 in \cite{Ga4}, we can easily have
\begin{equation}\label{eq4.1}
||\tilde{M}^{-1}||_\infty \leq ||\left(I +
(\tilde{B}^+)^{-1}\tilde{C}\right)^{-1} ||_\infty ||(\tilde{B}^+
)^{-1} ||_\infty \leq (n-1) ||(\tilde{B}^+ )^{-1} ||_\infty.
\end{equation}

Next, we give a upper bound for $ ||(\tilde{B}^+ )^{-1} ||_\infty$.
Note that $B^+$ is a Nekrasov matrix and $\tilde{B}^+= I -D + DB^+
$. By Lemma \ref{lem1}, $\tilde{B}^+$ is also a Nekrasov matrix. By
Theorem \ref{main1}, we easily get
\begin{equation}\label{eq4.2}
||(\tilde{B}^+)^{-1} ||_\infty  \leq \max\limits_{i\in N}
\frac{\eta_i(B^+)}{\min\{b_{ii}-h_i(B^+),1\}}.
\end{equation} From (\ref{eq4.1})  and  (\ref{eq4.2}), the
conclusion follows.
\end{proof}

\begin{example} \label{ex3} Consider the following matrix \[M =
\left[ \begin{array}{cccc}
  1  &\frac{1}{3}   &\frac{1}{3} &\frac{1}{2} \\
 \frac{1}{5}   &1   &- \frac{2}{5} & \frac{1}{5}   \\
 -1   &0    &1 &-\frac{1}{6}\\
 \frac{3}{4} & \frac{3}{4} &\frac{1}{2}  & 1
\end{array} \right].\]
It is not difficult to  check that  $M$ is not an $H$-matrix,
consequently, not a Nekrasov matrix, so we cannot use the bounds in
\cite{Ga0}, and bounds in Theorems \ref{G-Nek} and \ref{main1}.  On
the other hand, $M$ can be written $M=B^++C$ as in (\ref{dec}), with
\[B^+ =
\left[ \begin{array}{cccc}
  \frac{1}{2} &-\frac{1}{6}  &-\frac{1}{6} &0 \\
 0   &\frac{4}{5}    &-\frac{3}{5} &0   \\
  -1   &0    &1 &-\frac{1}{6}\\
 0 & 0 &- \frac{1}{4}  &  \frac{1}{4}
\end{array}\right],~ C =
\left[ \begin{array}{cccc}
  \frac{1}{2}  &\frac{1}{2}  &\frac{1}{2} &\frac{1}{2} \\
\frac{1}{5}    &  \frac{1}{5}    & \frac{1}{5}   & \frac{1}{5}    \\
 0   &0    &0 &0\\
\frac{3}{4} & \frac{3}{4}&\frac{3}{4} & \frac{3}{4}
\end{array} \right].\]
Obviously, $B^+$ is not strictly diagonally dominant and $M$ is not
a $B$-matrix, so we cannot apply the bound in \cite{Ga}.  However,
$B^+$ is a Nekrasov matrix and so $M$ is a $B$-Nekrasov matrix. The
diagonal matrix $W$ of Theorem \ref{B-N} is given by
\[W=diag\left(\frac{2}{3},\frac{3}{4},\frac{5}{6},\frac{5}{6}+\varepsilon \right)\]
with $ \varepsilon \in (0,\frac{1}{6})$. Hence, by Theorem \ref{B-N}
we can get the bound (\ref{bound_B_N1}) involved with
$\varepsilon\in (0,\frac{1}{6}) $ for $ \max\limits_{d\in
[0,1]^n}||(I-D+DM)^{-1}||_{\infty}$, which is drawn in Figure 2.
Meanwhile, by Theorem \ref{main2}, we can get the bound
(\ref{bound_B_N2}) for $ \max\limits_{d\in
[0,1]^n}||(I-D+DM)^{-1}||_{\infty}$, is 126.0000. It is easy to see
from Figures 2 and 3 that the bound in Theorem \ref{main2} is
smaller than that in Theorem \ref{B-N} (Theorem 2 in \cite{Ga4}).

\begin{figure}[tp]\label{Fig2} \centering
\includegraphics[width=4in]{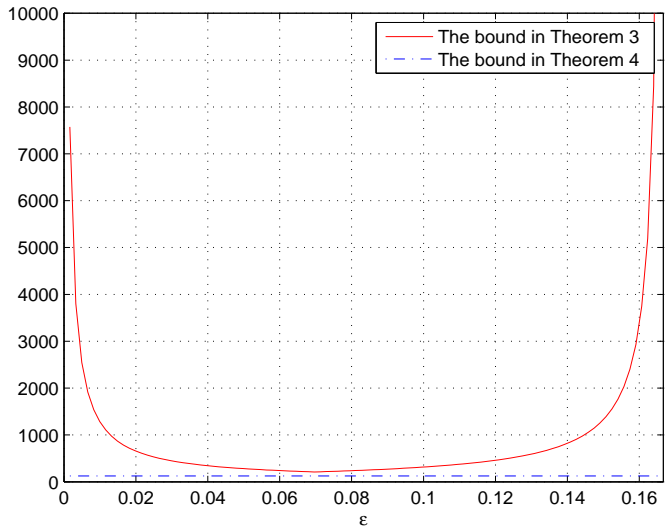}
\caption{The bounds in Theorems \ref{B-N} and \ref{main2}}
\end{figure}

\begin{figure}[tp]\label{Fig3} \centering
\includegraphics[width=4in]{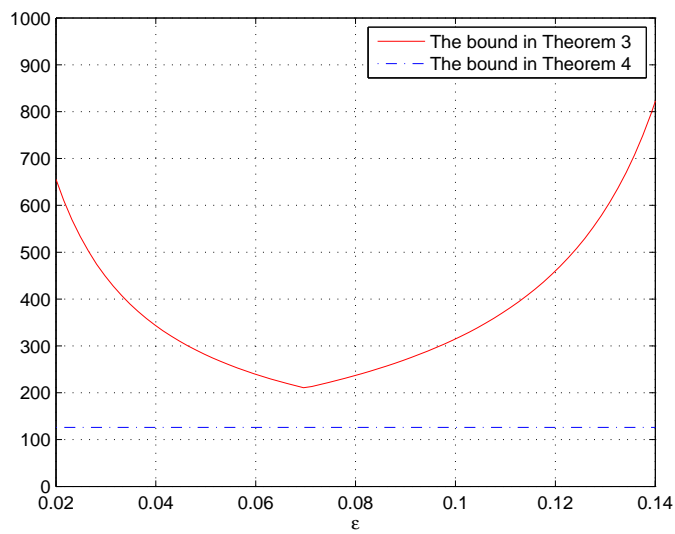}
\caption{The bounds in Theorems \ref{B-N} and \ref{main2} with
$\varepsilon \in [0.02,0.14]$ }
\end{figure}
\end{example}

\begin{example} \label{ex4} Consider the following matrix \[M =
\left[ \begin{array}{cccc}
  1  &\frac{1}{2}   &\frac{1}{2} &\frac{1}{2} \\
 \frac{1}{5}   &1   &- \frac{2}{5} & \frac{1}{5}   \\
 -1   &0    &1 &-\frac{1}{6}\\
 \frac{3}{4} & \frac{3}{4} &\frac{1}{2}  & 1
\end{array} \right].\]
And  $M$ can be written $M=B^++C$ as in (\ref{dec}), with
\[B^+ =
\left[ \begin{array}{cccc}
  \frac{1}{2} &0  &0 &0 \\
 0   &\frac{4}{5}    &-\frac{3}{5} &0   \\
  -1   &0    &1 &-\frac{1}{6}\\
 0 & 0 &- \frac{1}{4}  &  \frac{1}{4}
\end{array}\right],~ C =
\left[ \begin{array}{cccc}
  \frac{1}{2}  &\frac{1}{2}  &\frac{1}{2} &\frac{1}{2} \\
\frac{1}{5}    &  \frac{1}{5}    & \frac{1}{5}   & \frac{1}{5}    \\
 0   &0    &0 &0\\
\frac{3}{4} & \frac{3}{4}&\frac{3}{4} & \frac{3}{4}
\end{array} \right].\]
By computations,  \[h_1(B^+) =0, h_2(B^+) =\frac{3}{5},~h_3(B^+)
=\frac{1}{6},~h_4(B^+) =\frac{1}{24}.\] Obviously, $B^+$ is a
Nekrasov matrix and then $M$ is a  $B$-Nekrasov matrix. Since for
any $k>1$, $m_{1k}=r_1^+=\frac{1}{2} $, we cannot use the bound of
Theorem \ref{B-N} (Theorem 2 in \cite{Ga4}). However, by Theorem
\ref{main2}, we have
\[ \max\limits_{d\in [0,1]^n}||(I-D+DM)^{-1}||_{\infty} \leq  \frac{126}{5}.\]
\end{example}


\begin{acknowledgements}
 The authors would like to thank the
anonymous referees  and Prof. Lei Gao for their valuable suggestions
and comments to improve the original manuscript.

\end{acknowledgements}



\end{document}